\documentclass[12pt,a4paper]{amsart}
\usepackage{mathrsfs}
\usepackage{amssymb,amsmath,amsthm,color}
\usepackage{graphicx,mcite}
\usepackage{hyperref}
\usepackage{url}
\usepackage{setspace}
\newtheorem{theorem}{Theorem}[section]
\newtheorem{lemma}[theorem]{Lemma}
\newtheorem{proposition}[theorem]{Proposition}
\newtheorem{corollary}[theorem]{Corollary}

\theoremstyle{definition}

\newtheorem{remark}[theorem]{Remark}

\numberwithin{equation}{section}

\begin{document}

\title[real analytic expansion of the spectral projections]
{Real analytic expansion of spectral projection and extension of Hecke-Bochner identity}

\author{ Rajesh K. Srivastava}
\address{School of Mathematics, Harish-Chandra Research Institute, Allahabad, India 211019.}
\email{rksri@hri.res.in}


\date{\today}



\begin{abstract}
In this article, we review the Weyl correspondence of bigraded spherical harmonics
and use it to extend the Hecke-Bochner identities for the spectral projections
$f\times\varphi_k^{n-1}$ for function $f\in L^p(\mathbb C^n)$ with $1\leq p\leq\infty.$
We prove that spheres are sets of injectivity for the twisted spherical means with real
analytic weight. Then, we derive a real analytic expansion for the spectral projections
$f\times\varphi_k^{n-1}$ for function $f\in L^2(\mathbb C^n).$ Using this expansion we
deduce that complex cone can be a set of injectivity for the twisted spherical means.
\end{abstract}

\maketitle

\section{Introduction}\label{section1}
Weyl correspondence is a natural question of asking about an operator analogue of
the non-commutative polynomials on $\mathbb C^n.$ Let $\lambda\in\mathbb R\setminus\{0\}$
and $\mathbb H$ be a separable Hilbert space. Let
$W_1,\ldots,W_n,$ $ W_1^+,\ldots,W_n^+$ be unbounded operators on $\mathbb H$ satisfying
\[W_j^+=W_j^{\ast} \text{ and } \left[W_j^+, -W_j\right]=\frac{\lambda}{2}I,~ j=1,2,\ldots,n,\]
on a dense subspace $\mathcal D$ of $\mathbb H$ and all other commutators are zero.
Consider a polynomial $P(z)=z_1^2\bar z_1.$ In general, the question (asked by Weyl)
about the possible expression for $P(W,W^+),$ is still open. In 1984, Geller has
partially answered Weyl's question about operator analogue of the harmonic polynomials.
Let $z\in\mathbb C^n.$ For $\alpha, \beta\in \mathbb N^n$ and $P(z)=z^\alpha\bar z^\beta,$
define
\[\tau(P)=(W^+)^\beta W^\alpha \text { and } \tau'(P)=W^\alpha(W^+)^\beta.\]
The maps $\tau$ and $\tau'$ can be linearly extended to any polynomial on
$\mathbb C^n.$ In a long paper  \cite{Ge}, Geller prove that for any harmonic
polynomial $P$, $\tau(P)=\tau'(P).$ Using this, an analogue of Hecke-Bochner
identity for the Weyl transform has been derived, (see \cite{Ge}, p.645, Theorem 4.2).

\smallskip
A continuous function $f$ on $\mathbb R^n$  can be decomposed in terms of
spherical harmonics as
\begin{equation}\label{exp9}
f(x)=\sum_{k=0}^\infty a_{k,j}(\rho) Y_{k,j}(\omega),
\end{equation}
where $x=\rho\omega,~\rho=|x|,~\omega\in S^{n-1}$ and $\{Y_{kj}(\omega):
1,2,\ldots, d_k\}$ is an orthonormal basis for the space $V_k$ of the homogeneous
harmonic polynomials in $n$ variables of degree $k,$ restricted to the unit sphere
$S^{n-1},$ while the series in the right-hand side converges locally uniformly to $f.$
However, for more details, see \cite{SW}.

\smallskip

The well know Heche-Bochner identity says that the Fourier transform of any
piece in the above decomposition (\ref{exp9}) is preserved, (see \cite{SW}).
That is, the Fourier transform of $\tilde aP_k\in L^1 (\text{or}~ L^2),$ satisfies
$\widehat{{\tilde aP_k}}(x)=\tilde b(|x|)P_k(x),$ where $P_k$ is a solid spherical
harmonic of degree $k.$
An analogue of the Hecke-Bochner identity for the spectral projections
$f\times\varphi_k^{n-1}$ for function $f\in L^1 (\text{or }L^2),$
has been obtained, (see \cite{T}, p.70). Using the
Weyl correspondence of the spherical harmonics, we give a much simper proof of this
result. Further, we extend this result for $f\in L^p(\mathbb C^n)$ with
$1\leq p\leq\infty.$  As another application of the Weyl correspondence of the
spherical harmonics, we prove that sphere $S_R(o)=\{z\in\mathbb C^n: |z|=R\}$
is a set of injectivity for the twisted spherical means (TSM) with real analytic weight,
for the radial functions on $\mathbb C^n.$ An analogous for the Euclidean spherical
mean has been worked out in \cite{NRR}.

\smallskip

Since Laguerre function $\varphi_k^{n-1}$ is an eigenfunction of the special
Hermite operator $A=-\Delta_z+\frac{1}{4}|z|^2$ with eigenvalue $2k+n,$ the
projection $f\times\varphi_k^{n-1}$ is also an eigenfunction of $A$ with
eigenvalue $2k+n.$ As $A$ is an elliptic operator and eigenfunction of an
elliptic operator is real analytic \cite{N}, the projection $f\times\varphi_k^{n-1}$
must be a real analytic function on $\mathbb C^n.$ As a consequence of weyl
correspondence, we derive an important real analytic expansion for the spectral
projections $f\times\varphi_k^{n-1}$'s for function $ f\in L^2(\mathbb C^n),$
which we call Hecke-Bochner-Laguerre series for spectral projection. In the complex
plane, it is much simpler and it has been used in \cite{Sri2}, for proving a result
that any Coxeter system of even number of lines is a set of injectivity for the
twisted spherical means for $L^p(\mathbb C)$ with $1\leq p\leq2.$ However, question
of any Coxeter system of odd number of lines to be a set of injectivity for the
twisted spherical means is still unsolved.

\smallskip

As an application of expansion for the spectral projections $f\times\varphi_k^{n-1},$
we have shown that any complex cone is a set of injectivity for the twisted spherical
means for $L^p(\mathbb C^n), ~1\leq p\leq2$ as long as it does not lay on the level
surface of any bi-graded homogeneous harmonic polynomial on $\mathbb C^n.$ An analogous
result has been proved for the Euclidean spherical mean over the class of continuous
functions by Zalcman et al., (see \cite{AVZ}).

\section{Notation and Preliminaries}\label{section2}
We define the twisted convolution which arises in the study of
group convolution on Heisenberg group. The group $\mathbb H^n,$
as a manifold, is $\mathbb C^n \times\mathbb R$ with the group law
\[(z, t)(w, s)=(z+w,t+s+\frac{1}{2}\text{Im}(z.\bar{w})),~z,w\in\mathbb C^n\text{ and }t,s\in\mathbb R.\]
The  group convolution of  function $f,~g\in L^1(\mathbb H^n)$ is defined by
\begin{equation} \label{exp22}
f\ast g(z, t)=\int_{H^n}~f((z,t)(-w,-s))g(w,s)~dwds.
\end{equation}
An important technique in many problem on $\mathbb H^n$ is to take partial Fourier transform
in the $t$-variable to reduce matters to $\mathbb C^n$. Let
\[f^\lambda(z)=\int_\mathbb R f(z,t)e^{i \lambda t} dt\]
be the inverse Fourier transform of $f$ in the $t$-variable. Then a simple calculation shows
that
\begin{eqnarray*}
(f \ast g)^\lambda(z)&=&\int_{-\infty}^{~\infty}~f \ast g(z,t)e^{i\lambda t} dt\\
&=&\int_{\mathbb C^n}~f^\lambda (z-w)g^\lambda(w)e^{\frac{i\lambda}{2}\text{Im}(z.\bar{w})}~dw\\
&=&f^\lambda\times g^\lambda(z).
\end{eqnarray*}
Thus the group convolution  $f\ast g$ on the Heisenberg group can be
studied using the $\lambda$-twisted convolution $f^\lambda
\times_\lambda g^\lambda$ on $\mathbb C^n.$ For $\lambda \neq 0,$
a further scaling argument shows that it is enough to study the
twisted convolution for the case of $\lambda=1.$

\bigskip

We need the following basic facts from the theory of bigraded
spherical harmonics, (see \cite{D, Gr1, T} for details). We shall use
the notation $K=U(n)$ and $M=U(n-1).$ Then, $S^{2n-1}\cong K/M$ under
the map $kM\rightarrow k.e_n,$ $k\in U(n)$ and $e_n=(0,\ldots
,1)\in \mathbb C^n.$ Let $\hat{K}_M$ denote the set of all
equivalence classes of irreducible unitary representations of $K$
which have a nonzero $M$-fixed vector. It is known that for each
representation in $\hat{K}_M$ has a unique nonzero $M$-fixed vector,
up to a scalar multiple.

For a $\delta\in\hat{K}_M,$ which is realized on $V_{\delta},$ let
$\{e_1,\ldots, e_{d(\delta)}\}$ be an orthonormal basis of
$V_{\delta}$ with $e_1$ as the $M$-fixed vector. Let
$t_{ij}^{\delta}(k)=\langle e_i,\delta (k)e_j \rangle ,$ $k\in K$
and $\langle , \rangle$ stand for the innerproduct on $V_{\delta}.$
By Peter-Weyl theorem, it follows that $\{\sqrt{d(\delta
)}t_{j1}^{\delta}:1\leq j\leq d(\delta ),\delta\in\hat{K}_M\}$ is an
orthonormal basis of $L^2(K/M)$ (see \cite{T}, p.14 for details).
Define $Y_j^{\delta} (\omega )=\sqrt{d(\delta )}t_{j1}^{\delta}(k),$
where $\omega =k.e_n\in S^{2n-1},$ $k \in K.$ It then follows that
$\{Y_j^{\delta}:1\leq j\leq d(\delta ),\delta\in \hat{K}_M, \}$
forms an orthonormal basis for $L^2(S^{2n-1}).$

For our purpose, we need a concrete realization of the
representations in $\hat{K}_M,$ which can be done in the following
way. See \cite{Ru}, p.253, for details.
For $p,q\in\mathbb Z_+$, let $P_{p,q}$ denote the space of all
polynomials $P$ in $z$ and $\bar{z}$ of the form
$$
P(z)=\sum_{|\alpha|=p}\sum_{|\beta|=q}c_{\alpha\beta} z^\alpha
\bar{z}^\beta.
 $$
Let $H_{p,q}=\{P\in P_{p,q}:\Delta P=0\}.$ The elements of $H_{p,q}$ are
called the bigraded solid harmonics on $\mathbb C^n.$ The group  $K$
acts on $H_{p,q}$ in a natural way. It is easy to see that the space
$H_{p,q}$ is $K$-invariant. Let $\pi_{p,q}$ denote the corresponding
representation of $K$ on $H_{p,q}.$ Then representations in
$\hat{K}_M$ can be identified, up to unitary equivalence, with the
collection $\{\pi_{p,q}: p,q \in \mathbb Z_+\}.$

Define the bigraded spherical harmonic by $Y_j^{p,q}(\omega
)=\sqrt{d(p,q )}t_{j1}^{p,q}(k).$
Then $\{Y_j^{p,q}:1\leq j\leq d(p,q),p,q \in \mathbb Z_+ \}$ forms
an orthonormal basis for $L^2(S^{2n-1}).$ Therefore, for a continuous
function $f$ on $\mathbb C^n,$ writing $z=\rho \,\omega,$ where
$\rho
> 0$ and $\omega \in S^{2n-1},$ we can expand the function $f$ in
terms of spherical harmonics as
\begin{equation}\label{Bexp4}
f (\rho\omega) = \sum_{p,q\geq0} \sum_{j=1}^{d(p,q)} a_j^{p,q}(\rho)
Y_j^{p,q}(\omega),
\end{equation}
where the series on the right-hand side converges uniformly on every compact
set $K\subseteq\mathbb C^n.$ The functions $ a_j^{p,q} $ are called the
spherical harmonic coefficients of $f$.

\section{A review on the Weyl correspondence of spherical harmonics}\label{section3}
In this section, we revisit the Weyl correspondence of the spherical harmonics.
Then, we derive some important identities which describe the action of Weyl correspondence
of a harmonic polynomial to the Lagurrue functions $\varphi_k^{n-1}$'s.
We use these identities to give a very short and simple proof of the Hecke-Bochner
identities for the spectral projections. We would like to collect some of the
preliminaries from the work of Geller \cite{Ge}, where he has established the
Weyl correspondence of the spherical harmonics.

\smallskip

Define the symplectic Fourier transform $\mathcal F$ on Schwartz space $\mathcal S(\mathbb C^n)$ by
\begin{eqnarray*}
\mathcal F(f)(z)&=&\int_{\mathbb C^n}e^{-\frac{i}{2}\text{Im}(z.\bar{w})}f(w)dw\\
                &=&\int_{\mathbb C^n}\exp\left(\frac{-z.\bar w+\bar z.w}{4}\right)f(w)dw.
 \end{eqnarray*}
Let $\mathcal B(\mathbb H)$ be the space of all bounded linear operators on $\mathbb H.$
Since the operator $-i\left(-z.W^++\bar z.W\right)$ is essentially self-adjoint, therefore,
we can define the operator analogue of function $\mathcal F^{-1}(f)$ by an operator
 $\mathcal{G} : \mathcal S(\mathbb C^n)\rightarrow \mathcal B(\mathbb H),$ which is given by
 \[\mathcal{G}f=\int_{\mathbb C^n}\exp\left(\frac{-z.W^++\bar z.W}{4}\right)f(z)dz.\]
The operator $\mathcal G$ is known as the Weyl transform.
A composite operator $\mathcal{W}: \mathcal S(\mathbb C^n)\rightarrow \mathcal B(\mathbb H)$
which is defined by
\[\mathcal{W}(f)=\mathcal G\circ\mathcal F^{-1}(f),\]
is the Weyl correspondence of $f$. We state the following results which
are proved for the Weyl correspondence $\mathcal W$ of the spherical harmonics,
(see \cite{Ge}). We use these results for proving our main results.
\begin{lemma}\label{lemma1}\emph{\cite{Ge}}
If $P$ is a harmonic polynomial then
\[ \mathcal  W(P)=\tau(P)=\tau'(P).\]
\end{lemma}

\begin{lemma}\label{lemma2}\emph{\cite{Ge}}
Let $P,P_1\in P_{p,q}.$ Suppose there exist $\sigma\in U(n)$ such that $P=\pi(\sigma)P_1,$
where $\pi$ is an unitary representation of $U(n).$ Then
\[\mathcal{W}(P)=\left\{
  \begin{array}{ll}
    \pi(\bar\sigma)\mathcal{W}(P_1)\pi(\bar\sigma)^{\ast}, & \text{if } \lambda>0 ; \\
   \pi(\sigma)\mathcal{W}(P_1)\pi(\sigma)^{\ast}, & \text{if } \lambda<0,
  \end{array}
\right. \]
on the dense subspace $\mathcal D$ of ~$\mathbb H.$
\end{lemma}
Let us consider the following invariant differential operators which
arises in study of the twisted convolution on $\mathbb C^n.$
\[W_j=\widetilde Z_{j,\lambda}=\frac{\partial}{\partial z_j}-\frac{\lambda}{4}\bar z_j \text{ and }
W_j^+=\widetilde Z_{j,\lambda}^{\ast}=\frac{\partial}{\partial\bar z_j}+\frac{\lambda}{4} z_j, ~j=1,2,\ldots, n.\]
Let $P$ be a non-commutative homogeneous harmonic polynomial on $\mathbb C^n$ with expression
\[P(z)=\sum_{|\alpha|=p}\sum_{|\beta|=q}c_{\alpha\beta}z^\alpha\bar{z}^\beta.\]
Since $P$ is harmonic, by Lemma \ref{lemma1}, the operator $P(\widetilde Z)$ is the
Weyl correspondence $\mathcal W(P)$ of $P.$ Hence the operator
$P(\widetilde Z)$ can be expressed as
\[P(\widetilde Z)=\sum_{|\alpha|=p}\sum_{|\beta|=q}c_{\alpha\beta}{\widetilde {Z^{\ast}}}^\alpha\widetilde Z^\beta.\]
Next, we describe the action of operator
$P(\widetilde Z)$ to the Laguerre functions. For $k\in\mathbb Z_+,$ the Laguerre
functions $\varphi_k^{n-1}$ are defined by
\[\varphi_k^{n-1}(z)=L_k^{n-1}(\frac{1}{2}|z|^2)e^{-\frac{1}{4}|z|^2},\]
where $L_k^{n-1}$'s are the Laguerre polynomials of degree $k$ and order $n-1.$
We would like to refer \cite{BDJ, BGGV, Gr2}, for some study regarding Laguerre
calculus on the Heisenberg group.
In \cite{Ge}, Geller had proved an analogue of the Hecke-Bochner identity for the
Weyl transform of the type functions $\tilde aP\in L^1 (\text{or }L^2),$ which in
turn gives an analogous result for the spectral projections, (see \cite{NT1}). That is,
spectral projection of any type function is a type function. An application of the
following identities gives a very simple proof of the Hecke-Bochner identity for
the spectral projections.
\begin{lemma}\label{lemma3}\cite{Sri}
For  $P_1(z)=z_1^p\bar z_2^q\in H_{p,q},$ we have
\begin{equation}\label{exp3}
P_1(\widetilde Z)\varphi_k^{n-1}(z)=(-2)^{-p-q}P_1(z)\varphi_{k-p}^{n+p+q-1}(z),
\end{equation}
if ~$k\geq p$ and ~$0$ otherwise.
\end{lemma}

\begin{proof}
We have
\[\widetilde Z_1^{\ast}\varphi_{k}^{n-1}(z)=\left(\frac{\partial}{\partial\bar z_1}+\frac{1}{4}z_1\right)\varphi_{k}^{n-1}(z)\]
For $z\in\mathbb C^n$, let $z.\bar z=2t$. By chain rule $\frac{\partial}{\partial\bar z_1}=\frac{1}{2} z_1\frac{\partial}{\partial t}.$
Therefore,
\begin{eqnarray*}
\widetilde Z_1^{\ast}\varphi_{k}^{n-1}(z)&=&\left(\frac{1}{2}z_1\frac{\partial}{\partial t}+
\frac{1}{4}z_1\right)\left(L_{k}^{n-1}(t)e^{-\frac{1}{2}t}\right)\\
&=&\frac{1}{2}z_1\left(\frac{\partial}{\partial t}L_{k}^{n-1}(t)
-\frac{1}{2}L_{k}^{n-1}(t)+\frac{1}{2}L_{k}^{n-1}(t)\right)e^{-\frac{1}{2}t}.
\end{eqnarray*}
Since, the Laguerre's polynomials $L_k^n$'s satisfy the recursion relations
\begin{equation}\label{exp23}
\frac{d}{dx}L_k^n(x)=-L_{k-1}^{n+1}(x)\text{ and }L_{k-1}^{n+1}(x)+L_k^n(x)=L_k^{n+1}(x).
\end{equation}
Therefore, we can write
$\widetilde Z_1^{\ast}\varphi_{k}^{n-1}(z)=-\frac{1}{2}z_1\varphi_{k-1}^{n}(z).$
Similarly
\begin{eqnarray*}
\widetilde Z_2\varphi_{k}^{n-1}(z)&=&\left(\frac{1}{2}\bar z_2\frac{\partial}{\partial t}-
\frac{1}{4}\bar z_2\right)\left(L_{k}^{n-1}(t)e^{-\frac{1}{2}t}\right)\\
&=&\frac{1}{2}\bar z_2\left(\frac{\partial}{\partial t}L_{k}^{n-1}(t)
-\frac{1}{2}L_{k}^{n-1}(t)-\frac{1}{2}L_{k}^{n-1}(t)\right)e^{-\frac{1}{2}t}\\
&=&-\frac{1}{2}\bar z_2\varphi_{k}^{n}(z).
\end{eqnarray*}
Thus, we have
$\widetilde Z_1^{\ast}\widetilde Z_2\varphi_{k}^{n-1}(z)=2^{-2}z_1\bar z_2\varphi_{k-1}^{n+1}(z).$
Hence, by induction, we can conclude that
\[\widetilde{Z_1^{\ast}}^p\tilde Z_2^q\varphi_{k}^{n-1}(z)=(-2)^{-p-q}~z_1^p\bar z_2^q~\varphi_{k-p}^{n+1}(z).\]
\end{proof}

The identities (\ref{exp3}) has been used in an article \cite{Sri} for proving a result
that any sphere centered at the origin is a set of injectivity for the weighted twisted
spherical means on $\mathbb C^n.$ In this article, we generalize the identities (\ref{exp3})
for an arbitrary bigraded spherical harmonic $P\in H_{p,q}.$ We further extend the identities
(\ref{exp3}) for generalized Laguerre functions.

\smallskip

We use Lemma \ref{lemma3} to deduce the following identities which have been used frequently.
For $\lambda\in\mathbb R\setminus\{0\},$ let $\varphi_{k,\lambda}^{n-1}(z)=\varphi_{k}^{n-1}(|\lambda|z).$
\begin{theorem}\label{th5}
Let $P\in H_{p,q}$. Then
\begin{equation}\label{exp1}
P(\widetilde Z_\lambda)\varphi_{k,\lambda}^{n-1}=\left\{
  \begin{array}{ll}
   (-2\lambda)^{-p-q}P\varphi_{k-p,\lambda}^{n+p+q-1} , & \text{if } \lambda<0, ~k\geq q ; \\
   (-2\lambda)^{-p-q}P\varphi_{k-q,\lambda}^{n+p+q-1} , & \text{if } \lambda>0,  ~k\geq p.
  \end{array}
 \right.
\end{equation}
\end{theorem}

\begin{proof}
Using the scaling argument, it is enough to prove these identities for the cases $\lambda=\pm1.$
For $\lambda=1,$ we claim
\begin{equation}\label{exp2}
P(\widetilde Z)\varphi_{k}^{n-1}=(-2)^{-p-q}P\varphi_{k-p}^{n+p+q-1} ,\text{ if } ~k\geq p.
\end{equation}

We use the fact that any irreducible representation of a group is cyclic \cite{S}.
Since $H_{p,q}$ is an irreducible representation of $U(n)$, it follows that every non-zero
element of $H_{p,q}$ is a cyclic vector. Let $P_1(z)=z_1^p\bar z_2^q\in H_{p,q}.$
Then any $P\in H_{p,q}$ can be written as
\[P=\sum_{i=1}^m\pi(\sigma_i)P_1.\]
Since $\mathcal W$ is a linear operator, without loss of generality, we can assume that
$P=\pi(\sigma)P_1.$ In view of Lemma \ref{lemma2}, we can write
\[\mathcal{W}(P)=\pi(\sigma)\mathcal{W}(P_1)\pi(\sigma)^{\ast}.\]
Since $\varphi_k^{n-1}$ is $U(n)$-invariant, by Lemma \ref{lemma1}, we have
\[
P(\widetilde Z)\varphi_k^{n-1}=\pi(\sigma)P_1(\widetilde Z)\varphi_k^{n-1}.\]
An application of Lemma \ref{lemma3} gives
\begin{eqnarray*}
P(\widetilde Z)\varphi_k^{n-1}&=& (-2)^{-p-q}\pi(\sigma)\left(P_1\varphi_{k-p}^{n+p+q-1}\right) \\
             &=& (-2)^{-p-q}\pi(\sigma)P_1\varphi_{k-p}^{n+p+q-1}\\
             &=& (-2)^{-p-q}P\varphi_{k-p}^{n+p+q-1}.
\end{eqnarray*}
It is clear from the above computation and Lemma \ref{lemma3} that the proof for the case
$\lambda=-1$ is similar and hence we omit it here.
\smallskip

\noindent\textbf{Alternative proof:}
In the above proof, Lemma \ref{lemma2} played a crucial role. However, we give an alternative
proof of the identities (\ref{exp1}) without using Lemma \ref{lemma2}. Since the symplectic
Fourier transform $\mathcal F$ is invariant under $U(n)$ action, it follows that
\[\mathcal{F}^{-1}(P)=\mathcal{F}^{-1}(\pi(\sigma)P_1)=\pi(\sigma)\mathcal{F}^{-1}(P_1).\]
That is,
\[P\left(\frac{\partial}{\partial\bar z}, \frac{\partial}{\partial z}\right)\delta=
\pi(\sigma)P_1\left(\frac{\partial}{\partial\bar z}, \frac{\partial}{\partial z}\right)\delta
=\pi(\sigma)P_1(\tilde Z)\delta,\]
where $\delta$ is Dirac distribution at the origin.
Since, we know that $z_j\delta=0=\bar z_j\delta$, we can write
$P(\widetilde Z)\delta=\pi(\sigma)P_1(\widetilde Z)\delta.$ For more detail, see Geller \cite{Ge}, p. 625.
Denote \[T\delta=\left(P(\widetilde Z)-\pi(\sigma)P_1(\widetilde Z)\right)\delta=0.\]
For $z,w\in\mathbb C^n,$ define  $\tau_z(w)=(w-z)e^{\frac{i}{2}\text{Im}(z.\bar w)}.$
Then, it follows that $\tau_zT\delta=0$ on $\mathcal S(\mathbb C^n).$
For $\varphi=\varphi_k^{n-1},$ we have
\[\tau_zT\delta(\varphi)=\delta\left(\tau_zT\varphi\right)=
\tau_z\left(T\varphi\right)(0)=T\varphi(z)=0.\]
Hence
\begin{eqnarray*}
P(\widetilde Z)\varphi_k^{n-1}&=&\pi(\sigma)P_1(\widetilde Z)\varphi_{k}^{n-1} \\
                             &=& (-2)^{-p-q}\pi(\sigma)P_1\varphi_{k-p}^{n+p+q-1}\\
                             &=& (-2)^{-p-q}P\varphi_{k-p}^{n+p+q-1}.
\end{eqnarray*}
\end{proof}

These identities are quite useful and lead to a very short and simple proof of
the Hecke-Bochner identity for the spectral projections $f\times\varphi_k^{n-1}$
for function $f\in L^2(\mathbb C^n),$  (see \cite{T}, p.70). Using the identities
(\ref{exp1}), we prove that the boundary of any bounded domain is a set of injectivity
for certain weighted twisted spherical means for the radial functions on $\mathbb C^n.$

\begin{lemma} \label{lemma8}
Let $\tilde aP\in L^2(\mathbb C^n),$ where $\tilde a$ is a radial function and $P\in H_{p,q}$.
Then
\begin{equation}\label{exp34}
\tilde aP\times\varphi_k^{n-1}(z)=(2\pi)^{-n}P(z)~\tilde a\times\varphi_{k-p}^{n+p+q-1}(z),
\end{equation}
if ~$k\geq p$ and ~$0$ otherwise. The convolution in the right hand side is on
the space $~\mathbb C^{n+p+q}.$
\end{lemma}

\begin{proof}
Since the Laguerre functions $\left\{\varphi_{k-p}^{n+p+q-1}(r): k\geq p\right\}$
forms an orthonormal basis for $L^2\left(~\mathbb R^+,r^{2(n+p+q)-1}dr \right),$
therefore, we can express $\tilde a$ as
\[\tilde a=\sum_{k\geq p}C_{k-p}\varphi_{k-p}^{n+p+q-1}.\]
An application of the identities (\ref{exp1}) for $\lambda=1$ gives
\[\tilde aP=(-2)^{p+q}\sum_{k\geq p}C_{k-p}P(\widetilde Z)\varphi_{k}^{n-1}.\]
Let $j\geq p.$ Convolve both sides of the above equation by $\varphi_{j}^{n-1}.$ Then,
using the fact that $P(\widetilde Z)$ is a left invariant operator, we can write
\[\tilde aP\times\varphi_{j}^{n-1}=(-2)^{p+q}\sum_{k\geq p}C_{k-p}P(\widetilde Z)(\varphi_{k}^{n-1}\times\varphi_{j}^{n-1}).\]
Since the Laguerre functions satisfy the orthogonality conditions
$\varphi_{k}^{n-1}\times\varphi_{j}^{n-1}=(2\pi)^{n}\delta_{jk}\varphi_{k}^{n-1}.$
Therefore,
\begin{eqnarray*}
\tilde aP\times\varphi_{j}^{n-1}(z)&=&(2\pi)^{n}(-2)^{p+q}C_{j-p}P(\widetilde Z)\varphi_{j}^{n-1}(z)\\
&=&(2\pi)^{n}\left\langle\tilde a, \varphi_{j-p}^{n+p+q-1}\right\rangle P(z)\varphi_{j-p}^{n+p+q-1}(z)\\
&=&(2\pi)^{-n}P(z)~\tilde a\times\varphi_{j-p}^{n+p+q-1}(z).
\end{eqnarray*}
\end{proof}
Since the subspace $L^2\cap L^r(\mathbb C^n)$ is dense in $L^r(\mathbb C^n)$ for $1\leq r<\infty,$
therefore, the identities (\ref{exp34}) can be extended for $f\in L^r(\mathbb C^n).$ However, for functions
in $L^\infty(\mathbb C^n),$ we use the weighted functional equations for spherical function
$\varphi_k^{n-1}.$ The weighted functional equations can be obtained by considering the
Hecke-Bochner identity for the spectral projection of compactly supported functions.
For more details, see \cite{T}, p. 98.

\begin{lemma}\label{lemma3C}\cite{T}
For $z\in\mathbb C^n,$ let $P\in H_{p,q}$ and $d\nu_t=Pd\mu_t$. Then
\[\varphi_k^{n-1}\times\nu_t(z)=
(2\pi)^{-n}C(k,n+p+q)r^{2(p+q)}\varphi_{k-q}^{n+p+q-1}(t)P(z)\varphi_{k-q}^{n+p+q-1}(z),\]
if $k\geq q$ and ~$0$ otherwise.
\end{lemma}
Following result is an extension of the Hecke-Bochner identities for the spectral projections
for function $f\in L^r(\mathbb C^n)$ with $1\leq r\leq\infty.$

\begin{theorem}\label{th3}
Let $\tilde{a}P \in L^r(\mathbb C^n),1\leq r\leq\infty,~$ where $\tilde a$
is a radial function and $P\in H_{p,q}$. Then
\[\tilde aP\times\varphi_k^{n-1}(z)=(2\pi)^{-n}C(k, n+p+q)C_{k-p}^{n+p+q-1}P(z)\varphi_{k-p}^{n+p+q-1}(z),\]
if ~$k\geq p$ and ~$0$ otherwise.
\end{theorem}
\begin{proof}
Since $\tilde{a}P \in L^r(\mathbb C^n),$ for $1\leq r\leq\infty.$ By Young's inequality,
$\tilde{a}P\times\varphi_k^{n-1}$ exists and smooth. Consider the equation
$\overline{\tilde{a}P\times\varphi_k^{n-1}(z)}=\varphi_k^{n-1}\times \overline{\tilde{a}P}(z).$
Let $d\nu_t=\bar{P}d\mu_t$. By polar decomposition, we can write
\[\overline{\tilde{a}P\times\varphi_k^{n-1}(z)}=
\int_{t=0}^{\infty}\varphi_k^{n-1}\times d\nu_t(z)~\bar{\tilde a}(t)t^{2n-1}dt.\]
In view of Lemma \ref{lemma3C}, the above equation can be written as
\[\overline{\tilde{a}P\times\varphi_k^{n-1}(z)}=(2\pi)^{-n}C(k, \gamma)\bar{P}(z)\varphi_{k-p}^{\gamma-1}(z)
\int_{t=0}^{\infty}\varphi_{k-p}^{\gamma-1}(z)~\bar{\tilde a}(t)t^{2\gamma-1}dt,\]
where $\gamma=n+p+q.$ By taking the complex conjugate of both sides, we get
\begin{equation}\label{exp35}
\tilde{a}P\times\varphi_k^{n-1}(z)=(2\pi)^{-n}C(k,\gamma)P(z)\varphi_{k-p}^{\gamma-1}(z)
\int_{t=0}^{\infty}~\tilde{a}(t)\varphi_{k-p}^{\gamma-1}(t)t^{2\gamma-1}dt.
\end{equation}
This completes the proof of the theorem.
\end{proof}
\begin{remark}\label{rk1}
From (\ref{exp35}), we observe that the identities (\ref{exp34}) can be extended for more
general class of functions. An easy example is when $\tilde{a}$ to be a polynomial. Since
the integral in the right-hand side of (\ref{exp35}) exists, even when $\tilde{a}$ is a
radial tempered distribution, therefore, it is a feasible question to prove the identities
(\ref{exp34}) for $f$ to be a tempered distribution, which we leave open for the time being.
\end{remark}

As another application of the identities (\ref{exp1}), we prove that the boundary of
any bounded domain is a set of injectivity for certain weighted twisted spherical means,
for the radial functions on $\mathbb C^n.$ The weights here are spherical harmonics
on $S^{2n-1}.$ We use this to prove that sphere $S_R(o)$ is a set injectivity for the
twisted spherical means with real analytic weight, for the radial functions on
$\mathbb C^n.$
Since this article is more concerned about Weyl correspondence and its applications,
therefore, we skip here to write more histories of sets of injectivity for the twisted
spherical means. We would like to refer \cite{AR,NT1,Sri,Sri2,ST}.

\smallskip

In order to prove these results, we need the following result by Filaseta and Lam \cite{FL},
about the irreducibility of the Laguerre polynomials. Define the Laguerre polynomials by
\[L^\alpha_k(x)=\sum_{i=0}^k(-1)^i\binom{\alpha+k}{k-i}\frac{x^i}{i!},\]
$\text{ where } k\in \mathbb Z_{+} \text{ and } \alpha\in\mathbb C.$
\begin{theorem}\label{th6} \emph{\cite{FL}} Let $\alpha$ be a rational number, which
is not a negative integer. Then for all but finitely many $k\in\mathbb Z_+$, the polynomial
$L_k^\alpha(x)$ is irreducible over the field of rationals.
\end{theorem}
Using Theorem \ref{th6}, we have obtained the following corollary about the zeros of the
Laguerre polynomials.
\begin{corollary}\label{cor1}
Let $k\in\mathbb Z_{+}.$ Then for all but finitely many $k$, the Laguerre polynomials
$L^{n-1}_k$'s have distinct zeros over the reals.
\end{corollary}
\begin{proof}
 By Theorem \ref{th6}, there exists $k_o\in\mathbb Z_+$ such that $L_k^{n-1}$'s are
irreducible over $\mathbb Q$ whenever  $k\geq k_o.$ Therefore, we can find polynomials
$P_1, P_2\in\mathbb Q[x]$ such that $P_1L_{k_1}^{n-1}+ P_2L_{k_2}^{n-1}=1,$ over $\mathbb Q$
with $k_1,k_2\geq k_o.$ Since this identity continue to hold on $\mathbb R,$ it follows that
$L_{k_1}^{n-1}$ and $L_{k_2}^{n-1}$ have no common zero over $\mathbb R.$
\end{proof}

In order to prove Lemma \ref{lemma7}, we also need the following lemma from \cite{T} about
the spectral projections of a radial function.
\begin{lemma} \label{lemma1C}\emph{\cite{T}}
Let  $f$ be a radial function in $L^2(\mathbb C^n)$. Then
\[f\times\varphi^{n-1}_k(z)=B^n_k\left\langle f,\varphi^{n-1}_k\right\rangle\varphi^{n-1}_k
~\text{where}~B^n_k=\frac{k!(n-1)!}{(n+k-1)!}.\]
Thus $f$ can expressed as
\[f=(2\pi)^{-n}\sum_{k=0}^\infty B^n_k\left\langle f,\varphi^{n-1}_k\right\rangle\varphi^{n-1}_k.\]
\end{lemma}
Let $\partial\Omega$ be the boundary of a bounded domain $\Omega$ in $\mathbb C^n.$
Let $\mu_r$ be the normalized surface measure on the sphere
$S_r(o)=\{z\in\mathbb C^n: |z|=r\}$ in $\mathbb C^n.$ For $P\in H_{p,q},$ write
$d\nu_r=Pd\mu_r.$

\begin{lemma}\label{lemma7}
Let $f$ be a radial function on $\mathbb C^n$ such that $e^{\frac{1}{4}|z|^2}f(z)\in
L^p(\mathbb C^n),$ for $1\leq p\leq \infty$. Suppose $f\times\nu_r(z)=0,$
$~\forall ~z\in\partial\Omega$ and $~\forall ~r>0.$ Then $f=0$ a.e. on $\mathbb C^n.$
\end{lemma}

\begin{proof}
Given that $f\times\nu_r(z)=0,$ $~\forall ~z\in\partial\Omega$
and $~\forall ~r>0.$ By polar decomposition, we can write
\begin{equation}\label{exp7}
f\times P\varphi_{k-q}^{n+p+q-1}(z)=\int_{r=0}^{\infty}f\times\nu_r(z)\varphi_{k-q}^{n+p+q-1}(r){r}^{2n-1}dr=0,
\end{equation}
whenever $z\in\partial\Omega$ and $k\geq q.$
Let $\varphi_\epsilon$ be a smooth, radial compactly supported approximate identity
on $\mathbb C^n.$ Then $f\times\varphi_\epsilon\in L^1\cap L^\infty(\mathbb C^n)$
and in particular $f\times\varphi_\epsilon\in L^2(\mathbb C^n).$  Since $\varphi_\epsilon$
is radial, by Lemma \ref{lemma1C}, we can write
\[f\times\varphi_\epsilon\times\nu_r(z)=
\sum_{k\geq 0}B_k^n\left\langle\varphi_\epsilon,\varphi_k^{n-1}\right\rangle f\times\left(\varphi_k^{n-1}\times\nu_r\right)(z).\]
By Lemma \ref{lemma3C}, it follows that
$f\times\varphi_\epsilon\times\nu_r(z)=0, \forall ~k\geq q$ and $\forall ~z\in\partial\Omega.$
Thus, without loss of generality, we can assume $f\in L^2(\mathbb C^n).$
We use the identities (\ref{exp1}) for $\lambda=-1.$
Let
\[\tilde A_j=\widetilde Z_{j,-1}=\frac{\partial}{\partial z_j}+\frac{1}{4}\bar z_j \text{ and }
\tilde A_j^{\ast}=\widetilde Z_{j,-1}^{\ast}=\frac{\partial}{\partial \bar z_j}-\frac{1}{4}\bar z_j, ~j=1,2,\ldots, n.\]
Then the set $\{I, \tilde A_j,\tilde A_j^{\ast}:~j=1,2,\ldots, n\}$ generates
the space of all the right invariant differential operators for twisted convolution on
$\mathbb C^n.$ Using Theorem \ref{th5} for $\lambda=-1$, we get
\begin{equation}\label{exp6}
P(\tilde A)\varphi_k^{n-1}(z)=(-2)^{-p-q}P(z)\varphi_{k-q}^{n+p+q-1}(z),
\end{equation}
if ~$k\geq q$ and ~$0$ otherwise. In view of the identities (\ref{exp6})
and the fact that $P(\tilde A)$ is a right invariant operator,
equation (\ref{exp7}) gives
\[P(\tilde A)\left(f\times\varphi_{k}^{n-1}\right)(z)=0.\]
Since $f$ is radial, by Lemma \ref{lemma1C},
$f\times\varphi_{k}^{n-1}=B^n_k\left\langle f, \varphi_{k}^{n-1} \right\rangle\varphi_{k}^{n-1}.$
Therefore,
\[B^n_k\left\langle f, \varphi_{k}^{n-1} \right\rangle P(\tilde A)(\varphi_{k}^{n-1})(z)=0.\]
Once again using the identities (\ref{exp6}), we can write
\[\left\langle f, \varphi_{k}^{n-1} \right\rangle P(z)\varphi_{k-q}^{n+p+q-1}(z)=0,\]
whenever $z\in\partial\Omega$ and $k\geq q.$ Since $P$ is a non zero harmonic polynomial,
therefore, by the maximal principle for harmonic function, $P$ can not vanish on
$\partial\Omega.$ Thus,
\[\left\langle f, \varphi_{k}^{n-1}\right\rangle\varphi_{k-q}^{n+p+q-1}(z)=0,\]
whenever $z\in\partial\Omega$ and $k\geq q.$ Since $\varphi_{k-q}^{n+p+q-1}$
is radial, we have two cases. First, when $\partial\Omega$ is a sphere centered
at the origin. Let $\partial\Omega=\{z\in\mathbb C^n: |z|=R\}.$ In view of Corollary
\ref{cor1}, the functions $\varphi_{k-p}^{n+p+q-1}$'s can not have a common zero
except for finitely many $k\in\mathbb Z_+$ with $k\geq q.$ Thus, there exists
some $k_o\in\mathbb Z_+$ such that $\left\langle f, \varphi_{k}^{n-1}\right\rangle=0,$
$~\forall ~k\geq k_o\geq q.$ Hence $f$ is a finite linear combination of
$\varphi_{k}^{n-1}$'s. By the given decay condition on $f$, we conclude
that $f=0$ a.e. on  $\mathbb C^n$. On the other hand, when $\partial\Omega$
is of non-constant curvature (or sphere off centered the origin), we reach
the conclusion that
\[\left\langle f, \varphi_{k}^{n-1}\right\rangle\varphi_{k-q}^{n+p+q-1}(R)=0,\]
whenever $R\in(a, b)$ with $a<b$ and $k\geq q.$ This is possible only if
$\left\langle f, \varphi_{k}^{n-1}\right\rangle=0,$ $~\forall ~k\geq q.$
Thus, we infer that  $f=0$ a.e. on  $\mathbb C^n.$
\end{proof}

Using Lemma \ref{lemma7}, we prove the following partial result for injectivity
of the TSM with real analytic weight. Let $g$ be a real analytic function
on $\mathbb C^n.$ Let $d\nu_r=gd\mu_r$ and $S_R(o)=\{z\in\mathbb C^n: |z|=R\}.$

\begin{theorem}\label{th8}
Let $f$ be a radial function on $\mathbb C^n$ such that $e^{\frac{1}{4}|z|^2}f(z)\in
L^p(\mathbb C^n),$ for $1\leq p\leq \infty$. Suppose $f\times\nu_r(z)=0,$
$~\forall ~z\in S_R(o)$ and $~\forall ~r>0.$ Then $f=0$ a.e. on $\mathbb C^n.$
\end{theorem}

\begin{proof}
By the given conditions, we have $f\times g\mu_r(z)=0,~\forall ~z\in S_R(o)$ and $~\forall ~r>0.$
Since $g$ is continuous, its spherical harmonic components can be expressed as
\[g_{lm}(z)=d(p,q)\int_{U(n)}g(\sigma^{-1}z)t_{lm}^{p,q}(\sigma)d\sigma,\]
for $1 \leq l,m\leq d(p,q).$ As $f$ is a radial function, we can write
\begin{eqnarray*}
f\times g_{lm}\mu_r(z)&=&d(p,q)\int_{U(n)}t_{lm}^{p,q}(\sigma)\int_{S_r(o)}f(z-w)g(\sigma^{-1}w)
e^{\frac{i}{2}\text{Im}(z.\bar{w})}d\mu_r(w)d\sigma\\
&=&d(p,q)\int_{U(n)}t_{lm}^{p,q}(\sigma)\int_{S_r(o)}f(\sigma^{-1}z-w)g(w)
e^{\frac{i}{2}\text{Im}(\sigma^{-1}z.\bar{w})}d\mu_r(w)d\sigma\\
&=&d(p,q)\int_{U(n)}f\times g\mu_r(\sigma^{-1}z)t_{lm}^{p,q}(\sigma)d\sigma=0,\\
\end{eqnarray*}
for all $z\in S_R(o)$ and $~\forall ~r>0.$ Since, each of the component $g_{lm}$
of $g$ can also be expressed as $g_{lm}(z)=\tilde a_{lm}(|z|)P(z),$ where $P\in H_{p,q}.$
Therefore, it follows that $\tilde a_{lm}(r)f\times P\mu_r(z)=0,~\forall ~z\in S_R(o)$
and $~\forall ~r>0.$ Since $\tilde a_{lm}$ is real analytic, by Lemma \ref{lemma7},
we conclude that $f=0$ a.e. on $\mathbb C^n.$
\end{proof}

We observe that the similar identities as to the identities (\ref{exp1}) continue
to work for the generalized Laguerre functions, (see \cite{EMOT}). These are the
eigenfunctions of the special Hermite operator, (see \cite{AR}).
Let $\mathbb C_\sharp=\{a\in\mathbb C: -a\notin\mathbb Z_+, \Re(a)<1\}.$
The formal expansion of the generalized Laguerre function is
\[M_a^{n-1}(x)=\frac{\Gamma(n-a)}{\Gamma(1-a)\Gamma(n)}\sum_{s=0}^{\infty}\frac{a_sx^s}{(n+s-1)!s!}; ~x, a\in\mathbb C,\]
where $a_s=a(a+1)(a+2)\cdots(a+s-1).$
For $a=-k, ~k\in\mathbb Z_+$, the function $M_{-k}^{n-1}$ is the usual Laguerre
polynomial $L_k^{n-1}$ of degree $k$. For $a\in\mathbb C_\sharp,$ the function
$M_{a}^{n-1}$ has another representation
\[M_{a}^{n-1}(x)=e^x\sum_{i=0}^\infty(-1)^i\binom{n+a+i-1}{a}\frac{x^i}{i!}.\]
The following properties as similar to the Laguerre polynomial $L_k^{n-1},$ continue
to hold for the generalized Laguerre functions $M_{a}^{n-1}$'s.
\begin{lemma}\label{lemma4}
Let $a\in\mathbb C_\sharp.$ Then
\[\frac{d}{dx}M_a^n(x)=-M_{a-1}^{n+1}(x)\text{ and }M_{a-1}^{n+1}(x)+M_a^n(x)=M_a^{n+1}(x).\]
\end{lemma}
\begin{proof}
We have
\[\frac{d}{dx}M_{a}^{n}(x)=M_{a}^{n}(x)+ e^x\sum_{i=1}^\infty(-1)^i\binom{n+a+i}{a}\frac{x^{i-1}}{(i-1)!}.\]
Put $i=j+1.$ Then
\begin{eqnarray*}
\frac{d}{dx}M_{a}^{n}(x)
                       &=& M_{a}^{n}(x)-M_{a}^{n+1}(x)\\
                       &=&-M_{a-1}^{n+1}(x).
\end{eqnarray*}
\end{proof}
We use the above lemma to prove an analogue of the identities (\ref{exp1})
for the generalized Laguerre functions
$\varphi_a^{n-1}(z)=M_a^{n-1}\left(\frac{1}{2}|z|^2\right)e^{-\frac{1}{4}|z|^2}.$
These $\varphi_a^{n-1}$'s are precisely the eigenfunctions of the special
Hermite Operator $A=-\Delta_z+\frac{1}{4}|z|^2$ with eigenvalue $-2a+n.$ For the sake
of simplicity, we prove the result for the case $\lambda=1.$
\begin{lemma}\label{lemma5}
Let $a\in\mathbb C_\sharp$ and  $P_1(z)=z_1^p\bar z_2^q\in H_{p,q}.$ Then
\begin{equation}\label{exp4}
 P_1(\widetilde Z)\varphi_a^{n-1}(z)=(-2)^{-p-q}P_1(z)\varphi_{a-p}^{n+p+q-1}(z).
\end{equation}
\end{lemma}
\begin{proof} The proof of Lemma \ref{lemma5} is quite similar to the proof
of Lemma \ref{lemma3} and hence we omit it here.
\end{proof}

\begin{remark}\label{rk2}
$(a).$  For $a\in\mathbb C_\sharp,$ the function $\varphi_a^{n-1}$ satisfies the
growth condition $\varphi_a^{n-1}(z)\approx B_a|z|^{2(a-n)}e^{\frac{1}{4}|z|^2}$
as $ |z|\rightarrow\infty,$ (see \cite{AR}). This shows that we can not use Lemma
\ref{lemma2} to generalize Lemma \ref{lemma5} for an arbitrary $P\in H_{p,q}.$
However, as the proof is based on the process of point wise differentiation and
multiplication, it is enough to prove the identity (\ref{exp4}) for each compact
set in $\mathbb C^n.$

\smallskip

$(b).$ Let $f(z)=\tilde a(|z|) P(z),$ where $P\in H_{p,q}.$ Suppose $f$ satisfies
the condition $f(z)e^{(\frac{1}{4}+\epsilon)|z|^2}\in L^p(\mathbb C^n),$ for some
$\epsilon>0$ and $1\leq p\leq\infty$. Then, it is a feasible problem to ask for
an analogue of the Hecke-Bochner identity for $f\times\varphi_a^{n-1}.$

\smallskip

Since $\varphi_a^{n-1}$ is an eigenfunction of the operator $A$ with eigenvalue
$-2a+n,$ the function $f\times\varphi_a^{n-1}$ is also an eigenfunction of $A$
with eigenvalue $-2a+n.$ As $A$ is an elliptic operator and eigenfunction of an
elliptic operator is real analytic \cite{N}, the function $f\times\varphi_a^{n-1}$
must be a real analytic function on $\mathbb C^n.$
\end{remark}

\section{A real analytic expansion of the spectral projections}\label{section4}
In this section, we derive a real analytic expansion for the spectral projections
$f\times\varphi_k^{n-1}$ for function $f\in L^2(\mathbb C^n).$ For $n=1,$ we have
given its proof in \cite{Sri2}, which is much simpler and used for proving some
injectivity results for the twisted spherical means on $\mathbb C.$ A non-empty
set $C\subset\mathbb C^n~(n\geq2)$ satisfying $\lambda C\subseteq C,$ for all
$\lambda\in\mathbb C$ is called a complex cone. Using expansion of the spectral
projections, we have shown that any complex cone $C$ is a set of injectivity for
the twisted spherical means for $L^p(\mathbb C^n), ~1\leq p\leq2$ if and only if
$C\nsubseteq P^{-1}_{s,t}(0),$ $\forall ~s,t\in\mathbb Z_+.$

\smallskip

We need the following result about an estimate of the bi-graded spherical harmonics.

\begin{lemma}\label{lemma6}
Let $Y_{p,q}\in H_{p,q}.$ Then $\|Y_{p,q}\|_\infty\leq C(p+q+1)^{n-1}\|Y_{p,q}\|_2.$
\end{lemma}
\begin{proof}
Let $\omega\in S^{2n-1}.$ Then $Y_{p,q}$ can be expressed as
\[Y_{p,q}(\omega)=\sum_{|\alpha|=p}\sum_{|\beta|=q}c_{\alpha\beta} \omega^\alpha\bar{\omega}^\beta.\]
It has been shown in \cite{T}, p.66 that
\[\|Y_{p,q}\|_2=\left(\sum_{|\alpha|=p}\sum_{|\beta|=q}|c_{\alpha\beta}|^2\alpha!\beta!\right)^{\frac{1}{2}},\]
where $\dim H_{p,q}=d(p,q)$ is given by
\[d(p,q)=\frac{(p+n-2)!(q+n-2)!((p+n-1)(q+n-1)-pq)}{p!q!(n-1)!}.\]
Since $H_{p,q}$ is a finite dimensional space and the fact that all norms on a
finite dimension space are equivalent, we get the following estimate for the
bigraded spherical harmonics.
\begin{eqnarray*}
\|Y_{p,q}\|_{\infty}&\leq&
\sup_{\omega\in S^{2n-1}}\sum_{|\alpha|=p}\sum_{|\beta|=q}\left|c_{\alpha\beta} \omega^\alpha\bar{\omega}^\beta\right|
\leq\sum_{|\alpha|=p}\sum_{|\beta|=q}\left|c_{\alpha\beta}\right|\\
&\leq&\sqrt{d(p,q)}\left(\sum_{|\alpha|=p}\sum_{|\beta|=q}\left|c_{\alpha\beta}\right|^2\right)^{\frac{1}{2}}
\leq\sqrt{d(p,q)}~\|Y_{p,q}\|_2.
\end{eqnarray*}
By a simple computation, it can be shown that $\sqrt{d(p,q)}\leq C(p+q+1)^{n-1},$
which gives the result.

\end{proof}

\begin{theorem}\label{th4}
Let $f\in L^2(\mathbb C^n).$ Then $Q_k(z)=f\times\varphi_k^{n-1}(z)$ is real analytic and its real
analytic expansion can be written as
\begin{equation}\label{exp5}
Q_k(z)=\sum_{p=0}^{k}\sum_{q=0}^\infty P_{p,q}^{k}(z)\varphi_{k-p}^{n+p+q-1}(z).
\end{equation}
\end{theorem}

\begin{proof}
We can write the spherical harmonic decomposition of function $f$ as
\begin{equation}\label{exp24}
 f(z)=\sum_{p=0}^\infty\sum_{q=0}^\infty\sum_{j=1}^{d(p,q)}\tilde a_j^{p,q}(|z|)P_{p,q}^j(z).
\end{equation}
Since $f\in L^2(\mathbb C^n)$ and set $\left\{\varphi_{k-p}^{n+p+q-1}:~k\geq p\right\}$ form an orthonornal basis
for $L^2\left(r^{2(n+p+q)-1}dr\right),$ it follows that
\[\tilde a_j^{p,q}=\sum_{k\geq p}C_{k-p,j}^{p,q}\varphi_{k-p}^{n+p+q-1}.\]
In view of the identities (\ref{exp1}) for $\lambda=1,$ we have
\[\tilde a_j^{p,q}P_{p,q}^j=\sum_{k\geq p}(-2)^{p+q}C_{k-p,j}^{p,q}P_{p,q}^j(\widetilde Z)\varphi_{k}^{n-1}.\]
From (\ref{exp24}) and orthogonality relations
$\varphi_{k}^{n-1}\times\varphi_{j}^{n-1}=(2\pi)^{n}\delta_{jk}\varphi_{k}^{n-1},$ we conclude that
\begin{eqnarray*}
f\times\varphi_{k_0}^{n-1}
&=&\sum_{p=0}^{k_0}\sum_{q=0}^\infty\sum_{j=1}^{d(p,q)}(2\pi)^{n}C_{k_0-p,j}^{p,q}P_{p,q}^j\varphi_{k_0-p}^{n+p+q-1}\\
&=&\sum_{p=0}^{k_0}\sum_{q=0}^\infty P_{p,q}^{k_0}\varphi_{k_0-p}^{n+p+q-1}, \text{ where }P_{p,q}^{k_0}\in H_{p,q}.
\end{eqnarray*}
Now, look at the following concrete expression for the spectral projections as
\begin{equation}\label{exp8}
 Q_k(z)=\sum_{p=0}^{k}\sum_{q=0}^\infty P_{p,q}^{k}(z)\varphi_{k-p}^{n+p+q-1}(z),
\end{equation}
where the series on the right-hand side converges to $Q_k$ in $L^2(\mathbb C^n).$
Since the spaces $H_{p,q}$'s are orthogonal among themselves, it follows that
\begin{equation}\label{exp38}
\left\|Q_k\right\|_2^2=\sum_{p=0}^k\sum_{q=0}^\infty\left\|Y_{p,q}^k\right\|_2^2
\left\|\varphi_{k-p}^{\gamma-1}\right\|_2^2<\infty,
\end{equation}
where
\[\left\|\varphi_{k-p}^{\gamma-1}\right\|_2^2=\int_0^\infty\left|\varphi_{k-p}^{\gamma-1}(r)\right|^2r^{2\gamma-1}dr
=2^{\gamma-1}\frac{(k-p+\gamma-1)!}{(k-p)!}.\]
Thus from (\ref{exp38}), we get an estimate
\[\left\|Y_{p,q}^k\right\|_2\leq C\left(\frac{(k-p)!}{2^{\gamma-1}(k-p+\gamma-1)!}\right)^{\frac{1}{2}}.\]
Using Lemma \ref{lemma6}, we can write
\begin{equation}\label{exp37}
\|Y_{p,q}^k\|_{L^\infty}\leq C\frac{(p+q+1)^{n-1}}{\left({2^{\gamma-1}(k-p+\gamma-1)!}\right)^{\frac{1}{2}}}.
\end{equation}
In order to prove that the series (\ref{exp5}) converges uniformly on each compact set
$K\subseteq\mathbb C^n,$ it is enough to prove that the series
\[h_p(z)=\sum_{q=m}^{\infty}P_{p,q}^k(z)\varphi_{k-p}^{\gamma-1}(z),\]
converges uniformly over each ball $B_R(o)\in\mathbb C^n,$ where $m=n+k-2p+2$ and
$\gamma=n+p+q.$ Since $q\geq m,$ it follow that $\gamma-1\geq k-p+1.$ Hence
\begin{eqnarray*}
\left|\varphi_{k-p}^{\gamma-1}(z)\right|
&=&\left|\sum_{j=0}^k(-1)^{j}\binom{k-p+\gamma-1}{k-p-j}\frac{\frac{1}{2}|z|^2}{j!}e^{-\frac{1}{4}|z|^2}\right|\\
&\leq&\binom{k-p+\gamma-1}{k-p}\left|\sum_{j=0}^k\frac{\frac{1}{2}|z|^2}{j!}e^{-\frac{1}{4}|z|^2}\right|\\
&\leq& \frac{(k-p+\gamma-1)!}{(k-p)!(\gamma-1)!}~e^{\frac{1}{4}|z|^2}.
\end{eqnarray*}
Let $|z|\leq R.$  Then we can write
\begin{eqnarray*}
|h_p(z)|&\leq&\sum_{q=m}^{\infty}\|Y_{p,q}\|_{\infty}|z|^{p+q}\left|\varphi_{k-p}^{\gamma-1}(z)\right|\\
        &\leq& C|R|^p~e^{\frac{1}{4}|R|^2}\sum_{q=m}^{\infty}\frac{(p+q+1)^{n-1}(k-p+\gamma-1)!}
        {\left({2^{\gamma-1}(k-p+\gamma-1)!}\right)^{\frac{1}{2}}(\gamma-1)!}|R|^q\\
        &=&C|R|^p~e^{\frac{1}{4}|R|^2}\sum_{q=m}^{\infty}~b_q~|R|^q.\\
\end{eqnarray*}
By a straightforward calculation we get $\lim_{q\rightarrow\infty}\frac{b_{q+1}}{b_q}=0<1.$
Hence each of the function $h_p$ is real analytic on $\mathbb C^n.$ That is, the right-hand
side of (\ref{exp5}) is a real analytic function which agreeing to the real analytic function
$Q_k$ a.e. on $\mathbb C^n.$ Hence (\ref{exp5}) is a real analytic expansion of $Q_k.$
\end{proof}

\begin{remark}\label{rk3}
$(a).$ Since the spectral projections $Q_k=f\times\varphi_k^{n-1}$ is a real analytic
function for $f\in L^r(\mathbb C^n),~1\leq r\leq\infty,$ therefore, it is feasible to
ask the question of finding a real analytic expansion for $Q_k$ as similar to (\ref{exp5}).

$(b).$ Let $f\in L^2(\mathbb C^n)$ and for each $k\in\mathbb Z_+,$ the projection
$e^{\frac{1}{4}|z|^2}Q_k(z)$ is a polynomial. Then there exits $m=m(k)\in\mathbb Z_+$ such that
\[Q_k(z)=\sum_{p=0}^{k}\sum_{q=0}^m P_{p,q}^{k}(z)\varphi_{k-p}^{n+p+q-1}(z),~P_{p,q}^{k}\in H_{p,q}.\]
The space consists of functions $f$ is much larger than $U(n)$-finite functions in
$L^2(\mathbb C^n).$ For this class of functions, we can ask the following question.
Does $Q_k(z)=0,$ for all $z\in \mathbb C^{n-1}\times\mathbb R$ and $\forall ~k\in\mathbb Z_+,$
implies $f=0$ a.e. on $\mathbb C^n$? This question for the case $n=1,$ is simpler and has been
done by equating the highest degree term of $e^{\frac{1}{4}|z|^2}Q_k$ to zero, (see \cite{Sri2},
Theorem 3.4).
\end{remark}

Using expansion (\ref{exp5}) for the spectral projections, we deduce the following result.
\begin{proposition}\label{prop1}
Let $f\in L^p(\mathbb C^n),~1\leq p\leq2.$ Suppose $f\times\mu_r(z)=0,$
$~\forall ~z\in C$ and $~\forall ~r>0.$ Then $f=0$ a.e. on $\mathbb C^n$
if and only if $C\nsubseteq P^{-1}_{s,t}(0),$ $\forall ~s,~t\in\mathbb Z_+.$
\end{proposition}

\begin{proof}
Since $f\in L^p(\mathbb C^n),~1\leq p\leq2,$ therefore, by convolving $f$ with
a right and radial compactly supported smooth approximate identity, we can assume
$f\in L^2(\mathbb C^n).$
Given that $f\times\mu_r(z)=0,$ $~\forall ~z\in C$ and $~\forall ~r>0.$
By polar decomposition, we can write
\begin{equation}\label{exp10}
Q_k(z)=f\times \varphi_{k}^{n-1}(z)=\int_{r=0}^{\infty}f\times\mu_r(z)\varphi_{k}^{n-1}(r){r}^{2n-1}dr=0,
\end{equation}
$~\forall ~z\in C$ and $~\forall ~k\in\mathbb Z_+.$ Therefore, by Theorem \ref{th4} we get
\[Q_k(z)=\sum_{s=0}^{k}\sum_{t=0}^\infty P_{s,t}^{k}(z)\varphi_{k-s}^{n+s+t-1}(z)=0,\]
$~\forall ~z\in C$ and $\forall ~k\in\mathbb Z_+.$
Since cone $C$ is closed under complex scaling, therefore, for $z\in C,$ it implies
$re^{i\theta}z\in C,$ for all $r>0$ and $\theta\in\mathbb R.$ Thus,
\[Q_k(re^{i\theta}z)=\sum_{s=0}^{k}\sum_{t=0}^\infty r^{s+t}e^{i(s-t)\theta}P_{s,t}^{k}(z)\varphi_{k-s}^{n+s+t-1}(rz)=0,\]
$~\forall ~r>0$ and $~\forall ~\theta\in\mathbb R.$ Since $Q_k$ is real analytic in $r,$ therefore
by equating the coefficients of $1, r, r^2,\ldots$ to zero, we conclude that
\begin{equation}\label{exp11}
\sum_{s+t=\alpha, ~s\leq k}\binom{n+k+t-1}{k-s} e^{i(s-t)\theta}P_{s,t}^{k}(z)=0,
\end{equation}
$\forall~\alpha\in\mathbb Z_+.$ Using the fact that the set $\{e^{i\beta\theta}: ~\beta\in\mathbb Z\}$
is an orthogonal set and the sum vanishes over each of the diagonal $s+t=\alpha,~s\leq k,$
it follows that $P_{s,t}^k(z)=0,$ for all $s\leq k$ and $t\in\mathbb Z_+.$
Since the equation (\ref{exp11}) is valid for each $k\in\mathbb Z_+.$ Therefore, we infer that
for $z\in C,$ the projections $Q_k(z)=0,$ $\forall ~k\in\mathbb Z_+$ if and only if $P_{s,t}(z)=0,$
for all $s, ~t\in\mathbb Z_+.$ This forces $Q_k\equiv0,$  $\forall ~k\in\mathbb Z_+$
if and only if
$C \nsubseteq P^{-1}_{s,t}(0),$ $\forall ~s,~t\in\mathbb Z_+.$
As $f\in L^2(\mathbb C^n),$ we can expand $f$ in terms of  its spectral projections as
\begin{equation}\label{exp12}
f=(2\pi)^{-n}\sum_{k=0}^\infty f\times\varphi_k^{n-1},
\end{equation}
where series in the right hand side converges in $L^2(\mathbb C^n).$ This is known as
the special Hermite expansion, (see \cite{T}, p.58). Hence in view of (\ref{exp12}),
 we reach to the conclusion that
 $f=0$ a.e. on $\mathbb C^n$ if and only if $C\nsubseteq P^{-1}_{s,t}(0),$
$\forall ~s,~t\in\mathbb Z_+.$
\end{proof}

\begin{remark}
We have observed that Proposition \ref{prop1} works for the class of all continuous
functions on $\mathbb C^n.$ However, proof for the class of continuous functions is
different than that of $L^p(\mathbb C^n),~1\leq p\leq2,$ because the expansion
(\ref{exp5}) is valid only for the functions in $L^2(\mathbb C^n).$ Hence, its
proof should be presented in full detail elsewhere.
\end{remark}

\noindent{\bf Acknowledgements:}
The author wishes to thank E. K. Narayanan for several fruitful discussions during my visit
to IISc, Bangalore. The author would also like to gratefully acknowledge the support provided
by the University Grant Commission and the Department of Atomic Energy, Government of India.


\begin{thebibliography}{11}
\bibitem{AR} M. L. Agranovsky and R. Rawat, {\em Injectivity sets for spherical means on the Heisenberg group,}
J. Fourier Anal. Appl., 5 (1999), no. 4, 363--372.

\bibitem{AVZ} M. L. Agranovsky, V. V. Volchkov and L. A. Zalcman, {\em Conical uniqueness sets
for the spherical Radon transform,} Bull. London Math. Soc. 31 (1999), no. 2, 231–-236.

\bibitem{BDJ} C. Berenstein, Der-Chen Chang, J. Tie, {\em Laguerre calculus and its applications on the Heisenberg group,}
 AMS/IP Studies in Advanced Mathematics, 22.

\bibitem{BGGV} R. W. Beals, B. Gaveau, P. C.  Greiner, J. Vauthier, {\em The Laguerre calculus on the Heisenberg group,}
 II. Bull. Sci. Math. (2) 110 (1986), no. 3, 225–-288.

\bibitem{D} C. F. Dunkl, {\em Boundary value problems for harmonic functions on the Heisenberg group,}
 Canad. J. Math. 38 (1986), no. 2, 478–-512.

\bibitem{EMOT} A. Erdlyi, W. Magnus, F. Oberhettinger and F. G. Tricomi,
 {\em Tricomi and Francesco G. Higher transcendental functions} (Based on notes left by Harry Bateman),
Vol. I. McGraw-Hill, New York, 1953.

\bibitem{FL} M, Filaseta and  T-Y, Lam, {\em On the irreducibility of the generalized Laguerre polynomials,}
Acta Arith.  105 (2002),  no. 2, 177--182.

\bibitem{Ge} D. Geller, {\em Spherical harmonics, the Weyl transform and the Fourier transform on the Heisenberg group,}
Canad. J. Math.  36  (1984),  no. 4, 615--684.

\bibitem{Gr1} P. C. Greiner, {\em Spherical harmonics on the Heisenberg group,} Canad. Math. Bull. 23 (1980), no. 4, 383–-396.

\bibitem{Gr2} P. C. Greiner, {\em On the Laguerre calculus of left-invariant convolution (pseudodifferential)
operators on the Heisenberg group,}
Goulaouic-Meyer-Schwartz Seminar, 1980–-1981, Exp. No. XI, 40 pp., École Polytech., Palaiseau, 1981.

\bibitem{N} R. Narasimhan, {\em Analysis on real and complex manifolds,} North-Holland Publishing Co., Amsterdam, 1985.

\bibitem{NRR} E. K. Narayanan, R. Rawat and S. K.  Ray, {\em Approximation by $K$-finite functions in $L\sp p$ spaces,}
Israel J. Math.  161  (2007), 187--207.

\bibitem{NT1} E. K. Narayanan and S. Thangavelu, {\em Injectivity sets for spherical means on the
Heisenberg group,} J. Math. Anal. Appl. 263 (2001), no. 2, 565--579.

\bibitem{Ru} W. Rudin, {\em Function theory in the unit ball of ~$\mathbb C^n$,} Springer-Verlag,
 New York-Berlin, 1980.

\bibitem{ST} G. Sajith and S. Thangavelu, {\em On the injectivity of twisted spherical means on ~$\mathbb C^n$,}
 Israel J. Math. 122 (2001), 79--92.

\bibitem{Sri} R. K. Srivastava, {\em Sets of injectivity for weighted twisted spherical means and support theorems,}
J. Fourier Anal. Appl.,  18 (2012), no. 3, 592--608.

\bibitem{Sri2} R. K. Srivastava, {\em Coxeter system of lines are sets of injectivity for the twisted
spherical means on $\mathbb C$,} (communicated).
DOI:~{\color{blue}\href{http://arxiv.org/find/all/1/au:+srivastava_r_k/0/1/0/all/0/1}{arXiv:1103.4571v1}}

\bibitem{SW} E. M. Stein and G. Weiss, {\em Introduction to Fourier analysis on Euclidean spaces,}
 Princeton Mathematical Series, No. 32. Princeton University Press, Princeton, N.J., 1971.

\bibitem{S} M. Sugiura, {\em Unitary representations and harmonic analysis}, North-Holland Mathematical Library, 44.
North-Holland Publishing Co., Amsterdam; Kodansha, Ltd., Tokyo, 1990.

\bibitem{T} S. Thangavelu, {\em An introduction to the uncertainty principle}, Prog. Math. 217, Birkhauser, Boston (2004).

\end{thebibliography}
\end{document}